\documentclass{article}
\usepackage[top=1in,bottom=1in ,left=1in ,right=1in]{geometry}
\usepackage{fancyhdr}
\usepackage{amsmath,amsthm,amssymb}
\usepackage{enumitem}
\usepackage[hyperfootnotes=false]{hyperref}
\newtheorem{theorem}{Theorem}
\newtheorem{lemma}[theorem]{Lemma}

\theoremstyle{definition}
\newtheorem{definition}[theorem]{Definition}

\theoremstyle{remark}

\begin{document}
\begin{center}
{\large On the Covering Number of $U_3(q)$}\\
Michael Epstein
\end{center}

\let\thefootnote\relax\footnote{The contents of this article are based on a portion of the author's dissertation \cite{Dissertation}, submitted in partial fulfillment of the requirements for the degree of doctor of philosophy at Florida Atlantic University.}

\begin{abstract}
The \emph{covering number}, $\sigma(G)$, of a finite, noncyclic group $G$ is the least positive integer $n$ such that $G$ is the union of $n$ proper subgroups. Here we investigate the covering numbers of the projective special unitary groups $U_3(q)$, give upper and lower bounds for $\sigma(U_3(q))$ when $q \geq 7$, and show that $\sigma(U_3(q))$ is asymptotic to $q^6/3$ as $q \rightarrow \infty$.
\end{abstract}

\section{Introduction}

A collection $\mathcal{C}$ of proper subgroups of a group $G$ such that $G  = \bigcup \mathcal{C}$ is called a \emph{cover} of $G$. Any group with a finite, noncyclic homomorphic image admits a finite cover. The \emph{covering number}, $\sigma(G)$, of a group $G$ which admits a finite cover is defined to be the least positive integer $n$ such that $G$ has a cover consisting of $n$ subgroups, and a cover of this size is called a \emph{minimal cover} of $G$. Following \cite{Cohn1994}, we adopt the convention that $\sigma(G) = \infty$ for any group $G$ which does not admit a finite cover.

A number of results on covering numbers of finite groups can be found in \cite{Cohn1994}, including results on the covering numbers of finite nilpotent and supersolvable groups. In the same paper it is conjectured that the covering number of a finite, noncyclic, solvable group is one more than the order of a suitable chief factor. This conjecture was proven by Tomkinson in \cite{Tomkinson1997}. Specifically he proved:

\begin{theorem}[Tomkinson] \label{thm:solvable}
If $G$ is a finite, noncylic, solvable group, then $\sigma(G) = q+1$, where $q$ is the order of the smallest chief factor of $G$ with more than one complement.
\end{theorem}

In light of this result, attention has shifted to investigating the covering numbers of nonsolvable groups, particularly simple and almost simple groups. Many results can be found in \cite{Britnell2008, Britnell2011, Bryce1999, EMNP17, Holmes2006, Holmes2010, Kappe2010, Kappe2016, Maroti2005}. In this work, we consider the covering numbers of projective special unitary groups. In light of the fact that $U_2(q)$ is known to be isomorphic to the projective special linear group $L_2(q)$, it follows from the results in \cite{Bryce1999} that $\sigma(U_2(q)) = \frac{1}{2}q(q+1)$ if $q$ is even and $\sigma(U_2(q)) = \frac{1}{2}q(q+1)+1$ if $q$ is odd, with a few exceptions for small $q$ (which are handled individually in the same paper). Consequently, we will consider the covering number of $U_3(q)$. The covering number of $U_3(q)$ is known for $q\leq 5$: $U_3(2) \cong 3^2:Q_8$ has the Klein 4-group as a homomorphic image, and so has covering number 3, and the covering numbers of $U_3(3)$, $U_3(4)$, and $U_3(5)$ were determined to be 64, 1745, and 176 respectively in \cite{Garonzi2019} (and independently in \cite{Dissertation}). In this article we investigate the covering number of $U_3(q)$ for $q \geq 7$. The main result of this paper is the following theorem:

\begin{theorem} \label{thm:mainresult}
Let $q \geq 7$ be a prime power. Then
\[k(q)+q^3(q+1)^2(q-1)/3 \leq \sigma(U_3(q)) \leq q^4+q^2+1 - m(q) + q^3(q+1)^2(q-1)/3\]
where 
\[ k(q)  = \begin{cases} 
      1, & \text{if $q$ is a power of 3} \\
      1+q^3, & \text{otherwise},
   \end{cases}
\]
and
\[ m(q)  = \begin{cases} 
      q^4/2, & \text{if $q$ is a power of 2} \\
      q^3+2q^2-2q-1, & \text{if $q$ is a power of 3} \\
      nq^4/p, & \text{if $q$ is a power of a prime $p$ such that $p = 3n \pm 1 \geq 5$}. 
   \end{cases}
\]
In particular, 
\[1+q^3(q+1)^2(q-1)/3 \leq \sigma(U_3(q))\leq q^4-q^3-q^2+2q+2 +q^3(q+1)^2(q-1)/3\,.\]

\end{theorem}

An immediate consequence of this is that $\sigma(U_3(q)) \sim \dfrac{q^6}{3}$ as $q \rightarrow \infty$. We also prove:
\begin{theorem} \label{thm:mainresult2}
$\sigma(SU_n(q)) = \sigma(U_n(q))$.
\end{theorem}
In particular, this shows that the bounds given in Theorem \ref{thm:mainresult} for $\sigma(U_3(q))$ are valid for $\sigma(SU_3(q))$ as well.

\section{Preliminaries}
We assume the reader is familiar with basic group theory. Throughout this article we use the notation and terminology of \cite{DandF}. We also use the notation of \cite{Atlas1985} for simple groups and subgroup structures. The reader is assumed to have some understanding of projective geometry over finite fields. We refer the reader to \cite{Hirschfeld1998} as a reference, but recall a few basic facts here. Let $q$ be a prime power. We denote the field of order $q$ by $\mathbb{F}_q$, and write $\mathbb{F}_q^*$ for the set of nonzero elements of $\mathbb{F}_q$. If $n \geq 2$ is a positive integer, we denote the point of $PG(n-1,q)$ which corresponds to  a nonzero vector $x \in \mathbb{F}_q^n$ by $[x]$. If $S\subseteq \mathbb{F}_q^n$, we write $[S] = \{[x] : x \in S\setminus \{0\} \}$.

Every invertible linear map $T: \mathbb{F}_{q}^n \rightarrow \mathbb{F}_{q}^n$ induces a collineation $\tau$ of $PG(n-1, q)$, given by $\tau([x]) = [T(x)]$. The map which sends each invertible linear transformation to its induced collineation is a homomorphism from the general linear group $GL_n(q)$ into the automorphism group of $PG(n-1,q)$ whose kernel is the center of $GL_n(q)$, which consists of the scalar transformations in $GL_n(q)$.

For $\alpha \in \mathbb{F}_{q^2}$, let $\overline{\alpha} = \alpha^q$. If $n$ is a positive integer, a \emph{conjugate symmetric sesquilinear form} on $\mathbb{F}_{q^2}^n$ is a function $f: \mathbb{F}_{q^2}^n \times \mathbb{F}_{q^2}^n \rightarrow \mathbb{F}_{q^2}$, such that for all $x, y, z \in \mathbb{F}_{q^2}^n$ and $\alpha, \beta \in \mathbb{F}_{q^2}$, $f(\alpha x + \beta y, z) = \alpha f(x,z) +\beta f(y,z)$ and $f(y,x) = \overline{f(x,y)}$. We say that the vectors $x$ and $y$ are \emph{orthogonal} if $f(x,y) = 0$. For $S \subseteq \mathbb{F}_{q^2}^n$, we define $S^\perp = \{x \in \mathbb{F}_{q^2}^n : f(x,s) = 0 \text{ for all $s \in S$}\}$. We note that this is a subspace of $\mathbb{F}_{q^2}^n$ for any subset $S$. The form $f$ is \emph{degenerate} if there is a nonzero vector $x$ such $\{x\}^\perp = \mathbb{F}_{q^2}^n$. We call a vector $x$ \emph{isotropic} if $x \in \{x\}^\perp$.

It is known that for a nondegenerate congujate symmetric sesquilinear form $f$ on $\mathbb{F}_{q^2}^n$ there is an \emph{orthonormal basis} $\{b_1, \dots, b_n\}$ for $\mathbb{F}_{q^2}^n$ such that for all $1 \leq i,\ j\leq n$, $f(b_i,b_i) = 1$ and $f(b_i, b_j) = 0$ if $i \neq j$. In this basis, for all $\alpha_i, \beta_i \in \mathbb{F}_{q^2}^n$, $1 \leq i \leq n$,
\[f\left(\sum_{i=1}^n \alpha_i b_i ,\sum_{i=1}^n \beta_i b_i\right) = \sum_{i=1}^n \alpha_i \overline{\beta_i}.\]
Conversely, given any basis $\{b_1, \dots, b_n\}$, this formula defines a conjugate symmetric sesquilinear form on $\mathbb{F}_{q^2}^n$. One consequence of this is that any two conjugate symmetric sesquilinear forms on $\mathbb{F}_{q^2}^n$ are equivalent, differing only by a change of basis.

An \emph{isometry} of $f$ is an invertible linear map $T \in GL_n(q^2)$ such that $f(T(x), T(y)) = f(x,y)$ for all $x, y \in \mathbb{F}_{q^2}^n$. Under composition these form a group called the \emph{general unitary group} $GU_n(q)$. The \emph{special unitary group} $SU_n(q)$ is the normal subgroup consisting of the isometries with determinant 1. We may then define the \emph{projective general unitary group} $PGU_n(q)$ and the \emph{projective special unitary group} $U_n(q)$ (often denoted by $PSU_n(q)$ or $PSU_n(q^2)$) as the images of $GU_n(q)$ and $SU_n(q)$ under the homomorphism described above. The group $U_n(q)$ has order $\displaystyle \frac{q^{\frac{1}{2}n(n-1)}}{\gcd(n,q+1)} \prod_{i=2}^n(q^i-(-1)^i)$ (which is $q^3(q^3+1)(q^2-1)/\gcd(3,q+1)$ for $n=3$) and is known to be simple except when $(n,q) \in \{(2,2), (2,3), (3,2)\}$.

A \emph{polarity} of a projective geometry is an inclusion reversing permutation of the subspaces which has order two. A nondegenerate conjugate symmetric sesquilinear form $f$ on $\mathbb{F}_{q^2}^n$ gives rise a polarity $\perp$ of $PG(n-1, q^2)$, given by $[W]^\perp = [W^\perp]$ for any subspace $W$ of $\mathbb{F}_{q^2}^n$. We call the polarity induced in this manner by a conjugate symmetric sesquilinear form a \emph{unitary polarity}. A point $P$ of the projective space is said to be an \emph{absolute point} of the polarity $\perp$ if $P \in P^\perp$. It is easily seen that if $\perp$ is the unitary polarity corresponding to the form $f$, then a point $P = [x]$ of $PG(n-1,q^2)$ is an absolute point of $\perp$ if and only $x$ is a nonzero isotropic vector with respect to the form $f$. In particular, $PG(2, q^2)$ has $q^3+1$ absolute points of the unitary polarity $\perp$, and any line $\ell$ of $PG(2,q^2)$ has 1 or $q+1$ absolute points, according to whether $\ell^\perp$ is an absolute point or not. We note that the action of $U_3(q)$ on the absolute points of $\perp$ in $PG(2,q^2)$ is doubly transitive, and the action on the $q^2(q^2-q+1)$ nonabsolute points is transitive. We say a triangle $\Delta$ in a projectve plane is \emph{self-polar} with respect to a polarity $\perp$ if the polarity interchanges each vertex of the triangle with its opposite side. Note that the vertices of such a triangle are necessarily nonabsolute.

For the remainder of this article we assume that $p$ is a prime, $a$ is a positive integer, $q=p^a$, $f$ is a conjugate symmetric sesquilinear form on $\mathbb{F}_{q^2}^3$, and $\perp$ is the corresponding polarity on $PG(2,q^2)$.

\section{Proof of Theorem \ref{thm:mainresult}}
We note that for the purposes of determining a covering number of a finite group one need only consider covers consisting of maximal subgroups, and therefore we first consider the maximal subgroups of $U_3(q)$. The subgroup structure of $U_3(q)$ was investigated by  H. H. Mitchell \cite{Mitchell} and R. W. Hartley \cite{Hartley} for $q$ odd or even respectively. They prove:

\begin{theorem}[Mitchell] \label{thm:Mitchell}
Let $H$ be a subgroup of $U_3(q)$, $q$ odd. Then, $H$ is a subgroup of the stabilizer of a point and a line, a triangle, or an imaginary triangle, (items 1--4 below) or $H$ is one of the following subgroups:
\begin{enumerate}
\item the stabilizer of an absolute point, with order $\dfrac{q^3(q+1)(q-1)}{\text{gcd}(3,q+1)}$.
\item the stabilizer of a nonabsolute point, with order $\dfrac{q(q+1)^2(q-1)}{\text{gcd}(3,q+1)}$.
\item the stabilizer of a triangle (in $PG(2, q^2)$), with order $\dfrac{6(q+1)^2}{\text{gcd}(3,q+1)}$.
\item the stabilizer of an imaginary triangle (i.e. a triangle in $PG(2, q^6)$ but not $PG(2, q^2)$), with order $\dfrac{3(q^2 - q+1)}{\text{gcd}(3,q+1)}$.
\item the stabilizer of a conic, with order $q(q+1)(q-1)$.
\item $U_3(q_0)$, if $q = q_0^k$ with $k$ odd.
\item $PGU_3(q_0)$, if $q = q_0^k$, $k$ is odd, and 3 divides both $k$ and $q_0 +1$.
\item the Hessian groups of order 216 (if 9 divides $q+1$), 72 and 36 (if 3 divides $q+1$).
\item a group of order 168, if $-7$ is not a square in $\mathbb{F}_q$.
\item a group of order 360, if 5 is a square in $\mathbb{F}_q$ and $\mathbb{F}_q$ does not contain a primitive third root of unity.
\item a group of order 720, if $q = 5^k$ with $k$ odd.
\item a group of order 2520, if $q=5^k$ with $k$ odd.
\end{enumerate}
\end{theorem}
As noted by King in \cite{King}, the groups of orders 168, 360, 720, and 2520 are respectively isomorphic to $L_3(2)$, $A_6$, $A_6.2$, and $A_7$. If $q$ is even, the maximal subgroups of $U_3(q)$ are as follows:

\begin{theorem}[Hartley] \label{thm:Hartley}
Let $q$ be even and $H$ be a maximal subgroup of $U_3(q)$. Then $H$ is one of the following subgroups:
\begin{enumerate}
\item the stabilizer of an absolute point, with order $\dfrac{q^3(q+1)(q-1)}{\text{gcd}(3,q+1)}$.
\item the stabilizer of a nonabsolute point, with order $\dfrac{q(q+1)^2(q-1)}{\text{gcd}(3,q+1)}$.
\item the stabilizer of a triangle (in $PG(2, q^2)$), with order $\dfrac{6(q+1)^2}{\text{gcd}(3,q+1)}$.
\item the stabilizer of an imaginary triangle, with order $\dfrac{3(q^2 - q+1)}{\text{gcd}(3,q+1)}$.
\item $U_3(q_0)$, if $q = q_0^k$, with $k$ an odd prime.
\item $PGU_3(q_0)$, if $q= q_0^3 = 2^{3k}$ with $k$ odd.
\item a group of order 36, if $q = 2$.
\end{enumerate}
\end{theorem}
In the case that $q=2$, the group of order 36 has structure $3^2:4$. Since the maximal subgroups of $U_3(q)$ are described in geometric terms, it will be convenient to work with a classification of the elements of $U_3(q)$ based on the geometric objects they fix. We will consider the following three types of elements in $U_3(q)$:

\begin{description}
\item[Type 1:] Elements that fix an absolute point of $PG(2, q^2)$.
\item[Type 2:] Elements that fix a nonabsolute point of $PG(2, q^2)$ but do not fix an absolute point.
\item[Type 3:] Elements that do not fix any points of $PG(2, q^2)$.
\end{description}
Each element of $U_3(q)$ is of exactly one of these types, and we consider the elements of each type in turn.

\subsection{Elements of type 1}
We observe that the elements of the Sylow $p$-subgroups (where $p$ is the prime divisor of $q$) of $U_3(q)$ are of this type. The stabilizers in $U_3(q)$ of the absolute points of $PG(2,q^2)$ are the normalizers of the Sylow $p$-subgroup of $U_3(q)$. These Sylow $p$-subgroups are nonabelian groups of order $q^3$, with exponent $p$ if $p$ is odd, and exponent 4 if $p=2$. Distinct Sylow $p$-subgroups of $U_3(q)$ intersect trivially, so there are a total of $q^6-1$ nonidentity elements in the Sylow $p$-subgroups of $U_3(q)$.

\subsection{Elements of type 2}
Here we will prove that an element of $U_3(q)$ which fixes a nonabsolute point of $PG(2,q^2)$ but does not fix any absolute points will fix exactly three points of $PG(2,q^2)$ which are the vertices of a self-polar triangle. From this it follows that the elements of type 2 in $U_3(q)$ can be covered using any collection $S$ of stabilizers of nonabsolute points such that $S$ contains the stabilizer of at least one vertex from each self-polar triangle.\\

As in \cite{Ennola1962, Wall1963}, for a monic polynomial $g(x) = x^n + \alpha_{n-1}x^{n-1} + \dots \alpha_1 x+\alpha_0 \in \mathbb{F}_{q^2}[x]$ with $\alpha_0 \neq 0$ we define $\tilde{g}(x) = \overline{\alpha_0}^{-1}(\overline{\alpha_0}x^n +\overline{\alpha_1}x^{n-1} + \dots +\overline{\alpha_{n-1}}x +1)$. The following properties are easily proven.

\begin{lemma} \label{lem:properties of tilde}
Let $g,h \in \mathbb{F}_{q^2}[x]$ be monic polynomials with $g(0) \neq 0$ and $h(0) \neq 0$. Then,
\begin{enumerate}
\item $\widetilde{gh}(x)  = \tilde{g}(x) \tilde{h}(x)$, and
\item if $T \in GU_{n}(q)$ and $g(x)$ is the minimal polynomial for $T$, then $\tilde{g}(x) = g(x)$.
\end{enumerate}
\end{lemma}

The following lemma is Lemma 2 of \cite{Ennola1962}.

\begin{lemma} \label{lem:Ennola}
Let $g(x) \in \mathbb{F}_{q^2}[x]$ be irreducible and monic of degree $n$ with $g(0) \neq 0$. If $\tilde{g}(x) = g(x)$, then $n$ is odd and every root $\xi$ of $g(x)$ satisfies $\xi^{q^n+1} = 1$.
\end{lemma}

\begin{theorem} \label{thm:must fix a self polar triangle}
If $\Delta$ is a self-polar triangle in $PG(2,q^2)$, then there are elements of $U_3(q)$ which fix the three vertices of $\Delta$ and no other points of $PG(2,q^2)$. Conversely, if $\tau \in U_3(q)$ fixes a nonabsolute point but does not fix an absolute point of $PG(2,q^2)$, then $\tau$ fixes exactly three points of $PG(2,q^2)$ which form a self-polar triangle.
\end{theorem}

\begin{proof}
First, suppose that $\Delta = \{P_0, P_1, P_2\}$ is a self-polar triangle in $PG(2,q^2)$. Then there are vectors $e_0, e_1, e_2 \in \mathbb{F}_{q^2}^3$ with $P_i = [e_i]$ for $i=0,1,2$ and such that $\{e_0, e_1,e_2\}$ is an orthonormal basis for $\mathbb{F}_{q^2}^3$. Choose any $\lambda, \mu \in \mathbb{F}_{q^2}^*$ satisfying $\lambda^{q+1} = \mu^{q+1} = 1$ so that no two of $\lambda, \mu,$ and $(\lambda \mu)^{-1}$ are equal. This can always be done since we may choose, for example, $\lambda =1$ and $\mu$ to be any element of order $q+1$ in $\mathbb{F}_{q^2}^*$. Let $T$ be the linear transformation from $\mathbb{F}_{q^q}^3$ to itself defined by $T(e_0) = \lambda e_0$, $T(e_1) = \mu e_1$, and $T(e_2) = (\lambda \mu)^{-1} e_2$. It is easy to see that $T \in SU_3(q)$ and that that $T$ has no eigenvectors except for scalar multiples of the $e_i$. It follows that the collineation of $PG(2,q^2)$ induced by $T$ fixes $P_0, P_1,$ and $P_2$ and fixes no other points of $PG(2,q^2)$.

Suppose now that $\tau \in U_3(q)$ fixes a nonabsolute point $P$ of $PG(2,q^2)$, but does not fix any absolute points. Let $T \in SU_3(q)$ be such that the collineation of $PG(2,q^2)$ induced by $T$ is $\tau$. Since $P$ is a nonabsolute point of $PG(2,q^2)$ there is $e_0 \in \mathbb{F}_{q^2}^3$ such that $P= [e_0]$ and $f(e_0, e_0) = 1$. Note that $e_0$ is an eigenvector for $T$ corresponding to some eigenvalue $\lambda \in \mathbb{F}_{q^2}$ satisfying $\lambda^{q+1} = 1$, and that $\{e_0\}$ can be extended to an orthonormal basis $\{e_0,e_1,e_2\}$ for $\mathbb{F}_{q^2}^3$. Note that $\text{span}\{e_1,e_2\} = \{e_0\}^\perp$ is $T$-invariant, so in this basis the matrix for $T$ has the form

\[\begin{bmatrix}
\lambda & 0 & 0\\
0 & \alpha & \beta\\
0 & \gamma & \delta
\end{bmatrix}\]

where $\lambda(\alpha \delta - \beta \gamma) = 1$. The characteristic polynomial for $T$ is $c(x) = (x-\lambda)g(x)$, where $g(x) = x^2 -(\alpha +\delta)x +\alpha \delta - \beta \gamma$ is the characteristic polynomial for the restriction $T|_{\{e_0\}^\perp}$ of $T$ to $\{e_0\}^\perp$.

We will now show that $g(x)$ factors over $\mathbb{F}_{q^2}$ by contradiction. Suppose that $g(x)$ is irreducible. Then $g(x)$ is the minimal polynomial for $T|_{\{e_0\}^\perp}$, and $T|_{\{e_0\}^\perp}$ preserves the restriction of the form $f$ to $\text{span}\{e_1,e_2\}$, and so is an element of $GU_2(q)$. Thus by Lemma \ref{lem:properties of tilde}, $\tilde{g}(x) = g(x)$. Note also that $T|_{\{e_0\}^\perp}$ is invertible, so $g(0) \neq 0$. Now by Lemma \ref{lem:Ennola}, the degree of $g(x)$ is odd which contradicts the fact that the degree of $g(x)$ is two. Consequently $g(x)$ factors in $\mathbb{F}_{q^2}[x]$.

Let us write $g(x) = (x-\mu)(x-\nu)$, where $\mu, \nu \in \mathbb{F}_{q^2}$. Note that neither $\mu$ nor $\nu$ can be equal to $\lambda$ or else $T$ would have a 2-dimensional eigenspace and would fix an absolute point of $PG(2,q^2)$. We must also have $\mu \neq \nu$ as follows: $\mu$ is an eigenvlue for $T|_{\{e_0\}^\perp}$ so there is a nonzero vector $v \in \text{span}\{e_1,e_2\}$ such that $T(v) = \mu v$. Since $T$ fixes no absolute points, $v$ is nonisotropic. Now $T$ preserves $\{e_0\}^\perp \cap \{v\}^\perp$, which is a 1-dimensional subspace of $\mathbb{F}_{q^2}^3$. Let $w$ generate this subspace. Note that since $v$ is nonisotrpoic and $w \in \{v\}^\perp$, $w$ is not a scalar multiple of $v$. Moreover, $w$ is also an eigenvector for $T|_{\{e_0\}^\perp}$. If $T(w) = \mu w$, then $T$ would have a 2-dimensional eigenspace, which as noted above is not the case. Therefore we can conclude that $T(w) = \nu w$ and that $\nu \neq \mu$. Since $T$ fixes no absolute points, $w$ is nonisotropic. Now, $e_0, v,$ and $w$ are eigenvectors corresponding to distinct eigenvalues of $T$ and so $\{e_0,v,w\}$ is a basis for $\mathbb{F}_{q^2}^3$. Thus, $\{P, [v],[w]\}$ is a triangle in $PG(2,q^2)$, and since the vectors $e_0, v,$ and $w$ are pairwise orthogonal, this triangle is self-polar. Finally, since $T$ fixes exactly these three points of $PG(2,q^2)$ and no others, the same holds for $\tau$.

\end{proof}

\subsection{Elements of type 3}
Our first goal in this section is to prove that elements of type 3 in $U_3(q)$ actually exist. Since these elements fix no points of $PG(2,q^2)$, and are therefore not covered by the stabilizers in $U_3(q)$ of the points of $PG(2,q^2)$, our second goal is to find a geometric object that these elements do fix, so that these elements may be covered using the corresponding stabilizers. As will be seen below, the desired object is a self-polar triangle in $PG(2,q^6)$.

\begin{lemma} \label{lem:number theory}
Let $r$ be the largest prime divisor of $q^2-q+1$. If $q>2$, then $r \geq 5$ and $\gcd(r, q^3(q-1)(q+1)^2) = 1$.
\end{lemma}

\begin{proof}
Note that $q^2-q+1$ is odd and $q^2-q+1  = (q-2)(q+1) +3$ is divisible by 3 if and only if $q+1$ is, in which case $q^2-q+1 \equiv 3 \enspace (\text{mod }9)$.   Thus $q^2-q+1$ is not a power of three unless $q=2$. Therefore when $q>2$ the largest prime divisor $r$ of $q^2-q+1$ is at least 5. Since $r \geq 5$ and $\gcd(q^2-q+1, q^3(q-1)(q+1)^2) = \gcd(3, q+1)$, it follows that  $\gcd(r, q^3(q-1)(q+1)^2) = 1$.
\end{proof}

\begin{theorem} \label{thm:existence of elements not fixing a point}
If $q>2$ then $U_3(q)$ contains elements not fixing any points of $PG(2,q^2)$.
\end{theorem}

\begin{proof}
Let $r$ be as in the previous lemma. Then $r$ divides the order of $U_3(q)$, so by Cauchy's Theorem $U_3(q)$ contains an element $\rho$ of order $r$. Since $r$ divides neither the order of the stabilizers of the absolute points of $PG(2,q^2)$, nor the order of the stabilizers of the nonabsolute points, $\rho$ is not contained in a stabilizer of either type, and therefore fixes no point of $PG(2,q^2)$.
\end{proof}

Our next goal is to prove that each element of type 3 in $U_3(q)$ is contained in a cyclic subgroup of order $(q^2-q+1)/\gcd(3,q+1)$ inside the pointwise stabilizer of a triangle in $PG(2,q^6)$. Let $\tau \in U_3(q)$ be an element which does not fix any points of $PG(2,q^2)$, and let $T \in SU_3(q)$ be a linear transformation whose induced collineation is $\tau$. Note that $T$ has no eigenvalues in $\mathbb{F}_{q^2}$. Consequently, the characteristic polynomial for $T$ is irreducible of degree 3, and $T$ has three distinct eigenvalues in $\mathbb{F}_{q^6}$, which must be of the form $\alpha, \alpha^{q^2}, \alpha^{q^4}$ for some $\alpha$ which is in $\mathbb{F}_{q^6}$ but not in $\mathbb{F}_{q^2}$ such that $\alpha^{q^4+q^2+1} = 1$. We may choose a basis $B = \{b_0, b_1, b_2\}$ such that the matrix for $T$ in this basis is in rational canonical form . That is, in this basis the matrix for $T$ is 

\[\begin{bmatrix}
0 & 0 & 1\\
1 & 0 & \beta\\
0 & 1 & \gamma
\end{bmatrix}\,,\]
where $\beta = -\alpha^{q^2+1}-\alpha^{q^4+q^2}-\alpha^{q^4+1}$, and $\gamma  =\alpha + \alpha^{q^2}+\alpha^{q^4}$. Note that $B$ is also a basis for $\mathbb{F}_{q^6}^3$ over $\mathbb{F}_{q^6}$, and that we can extend $T$ to a linear transformation $\hat{T}$ from $\mathbb{F}_{q^6}^3$ to itself by $\hat{T}(\sum_{i=0}^2 \delta_i b_i) = \sum_{i=0}^2 \delta_i T(b_i)$. We may also extend the form $f$ to a conjugate symmetric sesquilinear form $\hat{f}$ on $\mathbb{F}_{q^6}^3$ by $\hat{f}(\sum_{i=0}^2 \delta_i b_i, \sum_{j=0}^2 \varepsilon_j b_j)  = \sum_{i,j=0}^2 \delta_i \varepsilon_j^{q^3} f(b_i, b_j)$. It is easy to verify that $\hat{T}$ is an isometry for the form $\hat{f}$. Let us define 

\begin{align*}
e_0 =& \ D^{-1}[(\alpha^{2q^4+q^2} - \alpha^{q^4+2q^2})b_0 + (\alpha^{2q^2} - \alpha^{2q^4})b_1 +(\alpha^{q^4} - \alpha^{q^2})b_2]\\
e_1 =& \ D^{-1}[(\alpha^{q^4 +2} - \alpha^{2q^4+1})b_0 + (\alpha^{2q^4} -\alpha^2)b_1 + (\alpha - \alpha^{q^4})b_2]\\
e_2 =& \ D^{-1}[(\alpha^{2q^2+1} - \alpha^{q^2+2})b_0 + (\alpha^2 - \alpha^{2q^2})b_1 + (\alpha^{q^2} - \alpha)b_2],
\end{align*}
where $D = (\alpha^{q^4} - \alpha^{q^2})(\alpha^{q^4} - \alpha)(\alpha^{q^2} - \alpha)$. Then $e_i$ is an eigenvector for $\hat{T}$ corresponding to the eigenvector $\alpha^{q^{2i}}$ for $i=0,1,2$. Evidently, $\hat{T}$ fixes each point $[e_i]$, $i=0,1,2$, but fixes no other points of $PG(2,q^6)$. Moreover, $\{e_0,e_1,e_2\}$ is a linearly independent set of of vectors in $\mathbb{F}_{q^6}^3$, and so the projective points $[e_0], [e_1]$, and $[e_2]$ in $PG(2,q^6)$ are not collinear and therefore form a triangle. Let $\ell_i$ be the line of $PG(2,q^6)$ which passes through the points of $\{[e_0], [e_1], [e_2]\}\setminus \{[e_i]\}$ for $i=0,1,2$. That is, $\ell_i$ is the side of the triangle $\Delta_\tau = \{[e_0], [e_1], [e_2]\}$ opposite vertex $[e_i]$. Clearly $\hat{T}$ fixes each line $\ell_i$, and since any collineation of a projective plane fixes equally many points and lines, it follows that these are the only lines of $PG(2,q^6)$ fixed by $\hat{T}$. We will now show that $\Delta_\tau$ is a self-polar triangle with respect to the polarity coresponding to $\hat{f}$. We will ultimately show that each of the $[e_i]$ is a nonabsolute point, but first we will show that an odd number of them are nonabsolute. Suppose that $[e_0]$ is an absolute point. Then $[e_0]^\perp$ is a line fixed by $\hat{T}$ which passes though $[e_0]$ and therefore must be one of $\ell_1$ or $\ell_2$. Suppose without loss of generality that $[e_0]^\perp = \ell_1$. Then $[e_0]$ is the unique absolute point on $\ell_1$, and $\ell_1$ passes through $[e_2]$, so we can conclude that $[e_2]$ is a nonabsolute point. Thus $[e_2]^\perp$ is a line fixed by $\hat{T}$ which does not pass through $[e_2]$, and consequently $[e_2]^\perp = \ell_2$. This forces $[e_1]^\perp$ to be $\ell_0$ which contains $[e_1]$ so $[e_1]$ is an absolute point. Thus exactly one or all three of the points $[e_0], [e_1]$, and $[e_2]$ are nonabsolute points.

Now, at least one of $[e_0], [e_1]$, and $[e_2]$ is a nonabsolute point. Relabeling the points if necesary, we may assume that $[e_0]$ is a nonabsolute point. Then,
\[\hat{f}(e_0, e_0)  = \hat{f}(\hat{T}(e_0), \hat{T}(e_0)) = \hat{f}(\alpha e_0, \alpha e_0) = \alpha^{q^3+1}\hat{f}(e_0,e_0)\,.\]
Since $\hat{f}(e_0,e_0) \neq 0$, it follows that $\alpha^{q^3+1} = 1$. Thus, since $\alpha^{q^4+q^2+1} = 1$ and $\gcd(q^3+1, q^4+q^2+1) = q^2-q+1$, we must have $\alpha^{q^2-q+1} =1$. We also must have

\[\hat{f}(e_0,e_1) = \hat{f}(\hat{T}(e_0), \hat{T}(e_1) = \hat{f}(\alpha e_0, \alpha^{q^2} e_1) = \alpha^{q^5+1}\hat{f}(e_0,e_1)\,,\]
from which it follows that $(\alpha^{q^5+1}-1)\hat{f}(e_0,e_1) = 0$. We claim that $\alpha^{q^5+1}$ does not equal 1. On the contrary, suppose that $\alpha^{q^5+1} = 1$. Then $\alpha^{\gcd(3,q+1)} = 1$, since $\alpha^{q^2-q+1} = 1$ and $\gcd(q^5+1, q^2-q+1) = \gcd(3,q+1)$. But every element of $\mathbb{F}_{q^6}^*$ of order dividing $\gcd(3,q+1)$ is in $\mathbb{F}_{q^2}^*$, and this contradicts the fact that $\alpha \notin \mathbb{F}_{q^2}$. Consequently, $\alpha^{q^5+1} \neq 1$ and we may conclude that $\hat{f}(e_0,e_1) = 0$. A similar line of reasoning shows that $\hat{f}(e_1,e_2) = \hat{f}(e_2,e_0) = 0$. Now, since $e_0,\  e_1$, and $e_2$ are linearly independent and pairwise orthogonal, at most one of $[e_0], [e_1]$, and $[e_2]$ is an absolute point. Since an even number of them are absolute points as proven above, we conclude that none of them are. Thus for $i=0,1,2$, $[e_i] \notin [e_i]^\perp$, and since $\{[e_0]^\perp, [e_1]^\perp, [e_2]^\perp\} = \{\ell_0, \ell_1, \ell_2\}$, we can conclude that $[e_i]^\perp = \ell_i$ for $i=0,1,2$, i.e. that $\Delta_\tau$ is a self-polar triangle.

Let $\zeta \in \mathbb{F}_{q^6}^*$ be an element of order $q^2-q+1$, and define $\hat{Z}$ to be the linear transformation from $\mathbb{F}_{q^6}^3$ to $\mathbb{F}_{q^6}^3$ defined by $\hat{Z}(e_i) = \zeta^{q^{2i}}e_i$ for $i=0,1,2$. It is easy to see that $\hat{Z}$ has determinant 1 and is an isometry with respect to $\hat{f}$, and moreover, that $\hat{T} \in \langle \hat{Z} \rangle$. Now the matrix for $\hat{Z}$ in the basis $\{b_0, b_1, b_2\}$ is  given by

\[A^{-1} \begin{bmatrix}
\zeta & 0 & 0\\
0 & \zeta^{q^2} & 0\\
0 & 0 & \zeta^{q^4}
\end{bmatrix}A\,,\]

where
\[
A=\begin{bmatrix}
1 & \alpha & \alpha^2\\
1 & \alpha^{q^2} & \alpha^{2q^2}\\
1 & \alpha^{q^4} & \alpha^{2q^4}
\end{bmatrix}\,.
\]

By a straightforward calculation, which we omit, one can verify that the entries of this matrix are all from $\mathbb{F}_{q^2}$. It follows that the restriction $Z$ of $\hat{Z}$ to $\mathbb{F}_{q^2}^3$ maps $\mathbb{F}_{q^2}^3$ to itself and is an element of $SU_3(q)$ of order $q^2-q+1$ which fixes $\Delta_\tau$ pointwise, and that $T \in \langle Z\rangle$. If $\sigma$ is the collineation of $PG(2,q^2)$ induced by $Z$, then $\sigma$ also fixes $\Delta_\tau$ pointwise, $\lvert \sigma \rvert = (q^2-q+1)/\gcd(3,q+1)$, and $\tau \in \langle \sigma\rangle$. Therefore we have proven the following theorem:

\begin{theorem} \label{thm:imaginary triangle}
If $\tau \in U_3(q)$ is an element that does not fix any points of $PG(2,q^2)$, then
\begin{enumerate}
\item $\tau$ fixes exactly three points of $PG(2,q^6)$ which form a triangle $\Delta_\tau$,
\item there is $\sigma \in U_3(q)$ which also fixes the three vertices of $\Delta_\tau$ such that
\begin{enumerate}
\item $\lvert \sigma \rvert  = (q^2-q+1)/\gcd(3,q+1)$, and
\item $\tau \in \langle \sigma \rangle$.
\end{enumerate}
\end{enumerate}
\end{theorem}

In particular, this theorem shows that every element of type 3 in $U_3(q)$ is contained in a cyclic subgroup generated by an element of order $(q^2-q+1)/\gcd(3,q+1)$, so in covering the elements of type 3 in $U_3(q)$ we need only consider how to cover those of order $(q^2-q+1)/\gcd(3,q+1)$. Our next goal is to prove Theorem \ref{thm:unique maximal subgroup}, which shows that the unique way to do this using maximal subgroups is by all of the stabilizers of the imaginary triangles fixed by the elements of type 3 in $U_3(q)$.

\begin{theorem} \label{thm:unique maximal subgroup}
If $q \geq 7$ and $\sigma \in U_3(q)$ has order $(q^2-q+1)/\gcd(3,q+1)$, then $\sigma$ is contained in a unique maximal subgroup of $U_3(q)$, namely the stabilizer of the imaginary triangle $\Delta_\sigma$.
\end{theorem}

\begin{proof}
We note that $(q^2-q+1)/\gcd(3,q+1)$ divides neither $q^3(q^2-1)/\gcd(3,q+1)$, the order of the stabilizers of the absolute points of $PG(2,q^2)$, nor $q(q+1)^2(q-1)/\gcd(3,q+1)$, the order of the stabilizers of the nonabsolute points, so $\sigma$ fixes no points of $PG(2,q^2)$ and is therefore of type 3. By Theorem \ref{thm:imaginary triangle}, $\sigma$ is contained in the stabilizer $S$ of a unique imaginary triangle $\Delta_\sigma$. It is not immediately obvious that $S$ is a maximal subgroup of $U_3(q)$ when $q \geq 7$ considering the fact that the stabilizer of an imaginary triangle is not maximal in $U_3(5)$. However we will show both that $S$ is maximal and that $S$ is the unique maximal subgroup of $U_3(q)$ containing $\sigma$ when $q \geq 7$ by showing that $S$ is the only group from the lists of subgroups in Theorems \ref{thm:Mitchell} and \ref{thm:Hartley} which contains $\sigma$. Thus far we have already established that $\sigma$ is not contained in the stabilizer of any point, nor in the stabilizer of any imaginary triangle other than $\Delta_\sigma$. To see that $\sigma$ is not contained in the stabilizer of a triangle (in $PG(2,q^2)$) or the stabilizer of a conic, observe that $(q^2-q+1)/\gcd(3,q+1)$ divides neither $6(q+1)^2/\gcd(3,q+1)$ nor $q(q^2-1)$, the order of the stabilizer of a triangle or conic respectively. Nor does $(q^2-q+1)/\gcd(3,q+1)$ divide $3q_0^3(q_0^3+1)(q_0^2-1)$ where $q = q_0^k$, $k$ is odd and $k \geq 3$, and consequently $\sigma$ is not contained in a maximal subgroup of $U_3(q)$ isomorphic to $U_3(q_0)$ or $PGU_3(q_0)$. This completes the proof if $q$ is even.

For the odd case, we must show that $\sigma$ is not contained in any of the subgroups from items 8-12 of the list in Theorem \ref{thm:Mitchell}. We do this by showing that $\lvert \sigma \rvert$ does not divide any of the orders of these subgroups, which are 36, 72, 168, 216, 360, 720, and 2520. If $q \geq 88$, then $(q^2-q+1)/\gcd(3,q+1) > 2520$, so we need only consider the case that $7 \leq q < 88$. In this case, $q \neq 5^k$ with $k$ odd, so we need not consider the subgroups isomorphic to $A_6.2$ or $A_7$. If $q \geq 34$, then $(q^2-q+1)/\gcd(3,q+1) >360$, and the theorem holds in this case also. It is easily checked that $(q^2-q+1)/\gcd(3,q+1)$ does not divide 36, 72, 168, 216, or 360 in any of the remaining cases, i.e. when $q \in \{7, 9, 11, 13, 17, 19, 23, 27, 29, 31\}$, which completes the proof.
\end{proof}

Finally, we note that the stabilizers of the imaginary triangles fixed by the elements of type 3 in $U_3(q)$ are the normalizers of Sylow subgroups of $U_3(q)$ and are therefore all conjugate.

\begin{lemma} \label{lem:sylow normalizer}
Let $q > 2$ and $\tau \in U_3(q)$ be an element that does not fix any points of $PG(2,q^2)$. Then the stabilizer of $\Delta_\tau$ in $U_3(q)$ is the normalizer of a Sylow $r$-subgroup of $U_3(q)$, where $r$ is the largest prime divisor of $q^2-q+1$.
\end{lemma}

\begin{proof}
Let $S$ be the stabilizer of $\Delta_\tau$ in $U_3(q)$. Note that $S$ has order $3(q^2-q+1)/\gcd(3,q+1)$ by Theorem \ref{thm:Mitchell} if $q$ is odd or by Theorem \ref{thm:Hartley} is $q$ is even, and by Theorem \ref{thm:imaginary triangle}, there is $\sigma \in S$ of order $(q^2-q+1)/ \gcd(3,q+1)$. The index of $\langle \sigma \rangle$ in $S$ is 3, which is the minimal prime divisor of the order of $S$, so $\langle \sigma \rangle$ is normal in $S$. Now, $\langle \sigma \rangle$ contains a unique Sylow $r$-subgroup $R$ of $U_3(q)$ which is a characteristic subgroup of $\langle \sigma\rangle$. Since $\langle \sigma \rangle$ is normal in $S$, it follows that $R$ is also normal in $S$.\\

On the other hand, suppose that $\nu \in U_3(q)$ normalizes $R$. Notice that $R$ is cyclic, and let $\rho$ be a generator for $R$. Then for some integer $k$ and each vertex $P$ of $\Delta_\tau$, $\rho \nu (P) = \nu \rho^k(P) = \nu(P)$. So $\rho$ fixes each of the points $\nu(P)$, where  $P \in \Delta_\tau$, but the only points of $PG(2, q^6)$ fixed by $\rho$ are the vertices of $\Delta_\tau$, so it follows that $\nu$ permutes these and thus $\nu \in S$.
\end{proof}

It follows from the previous lemma that the stabilizers of the imaginary triangles fixed by the elements of type 3 in $U_3(q)$ are self-normalizing and form a single conjugacy class of size $q^3(q+1)^2(q-1)/3$.

\subsection{Bounds on the covering number}
We are nearly ready to establish the upper and lower bounds on $\sigma(U_3(q))$ given in Theorem \ref{thm:mainresult}. We will need to make use of some known results on polarity graphs, as defined below.

\begin{definition}
Let $\Pi$ be a projective plane and $\perp$ be a polarity on $\Pi$. The \emph{polarity graph} associated to the pair $(\Pi, \perp)$, is a graph $\mathcal{G}$ whose vertices are the points of $\Pi$, where there is an edge from $P$ to $Q$ if and only if $P \in Q^\perp$.
\end{definition}

We note that the triangles in the polarity graph $\mathcal{G}$ correspond to the self-polar triangles in $\Pi$. We will need the following result, which is a combination of Theorems 1.3, 4.7, and Remark 4.6 of \cite{Mattheus2018}.

\begin{theorem} \label{thm:triangle free set}
Let $\perp$ be a unitary polarity on $PG(2,q^2)$, and let $\mathcal{G}$ be the corresponding polarity graph. Then there exists a set $S$ of $m(q)$ nonabsolute points of $PG(2,q^2)$ such that the subgraph of $\mathcal{G}$ induced by $S$ is triangle-free, where
\[ m(q)  = \begin{cases} 
      q^4/2, & \text{if $q$ is a power of 2} \\
      q^3+2q^2-2q-1, & \text{if $q$ is a power of 3} \\
      kq^4/p, & \text{if $q$ is a power of a prime $p$ such that $p = 3k \pm 1 \geq 5$}. 
   \end{cases}
\]
\end{theorem}

With this resut on triangle-free subsets in unitary polarity graphs, we can now prove the upper and lower bounds on $\sigma(U_3(q))$ for $q \geq 7$ given in Theorem \ref{thm:mainresult}:

Let $S$ be a set as in the previous theorem, and let $S'$ be the set of all nonabsolute points of $PG(2, q^2)$ which are not in $S$. Note that $\lvert S' \rvert = q^4-q^3 +q^2 -m(q)$. Since the subgraph of $\mathcal{G}$ induced by $S$ is triangle-free, every self-polar triangle in $PG(2,q^2)$ has at least one vertex in $S'$. By Lemma \ref{thm:must fix a self polar triangle}, every element of $U_3(q)$ which fixes a nonabsolute point of $PG(2,q^2)$ but not an absolute point is contained in the stabilizer in $U_3(q)$ of at least one point of $S'$. Consequently the set $\mathcal{C}$ consisting of
\begin{enumerate}
\item the $q^3+1$ stabilizers of the absolute points of $PG(2,q^2)$, 
\item the $q^4-q^3 +q^2 -m(q)$ stabilizers of the nonabsolute points from $S'$, and
\item the stabilizers of the $q^3(q+1)^2(q-1)/3$ imaginary triangles fixed by the elements of type 3 in $U_3(q)$
\end{enumerate}
is a cover of $U_3(q)$ with $\lvert \mathcal{C} \rvert = q^4+q^2+1 - m(q) + q^3(q+1)^2(q-1)/3$. Note that for $q \geq 7$, $m(q) \geq q^3+2q^2-2q-1$, so $\lvert \mathcal{C} \rvert \leq q^4-q^3-q^2+2q+2 +q^3(q+1)^2(q-1)/3$. This establishes the upper bound given in Theorem \ref{thm:mainresult}.

To justify the lower bound, let $\mathcal{C}$ be a minimal cover for $U_3(q)$. We may without loss of generality assume that $\mathcal{C}$ consists of maximal subgroups of $U_3(q)$. Note that by Theorem \ref{thm:unique maximal subgroup} and the comments following the proof of Lemma \ref{lem:sylow normalizer}, $\mathcal{C}$ must contain the stabilizers of all $q^3(q+1)^2(q-1)/3$ imaginary triangles fixed by the elements of $U_3(q)$ that fix no point of $PG(2,q^2)$. Since these subgroups do not contain, for example, any of the involutions in $U_3(q)$, we must have $\lvert \mathcal{C} \rvert >q^3(q+1)^2(q-1)/3$.

For the remainder of the proof we will assume that the prime divisor $p$ of $q$ is not equal to 3. Let $\Omega$ be the set of nonidentity elements contained in the union of all of the Sylow $p$-subgroups of $U_3(q)$. We will partition the maximal subgroups of $U_3(q)$ into sets $X_i$. Specifically, let $X_1$ be the set of stabilizers in $U_3(q)$ of the absolute points of $PG(2,q^2)$, $X_2$ be the set of stabilizers of the nonabsolute points, $X_3$ be the set consisting of the stabilizers of the triangles in $PG(2,q^2)$, $X_4$ be the set of the stabilizers of the imaginary triangles fixed by the elements of type 3 in $U_3(q)$, and $X_5$ be the set of subgroups of $U_3(q)$ isomorphic to $U_3(q_0)$ or $PGU_3(q_0)$ where $q = q_0^k$ with $k$ odd and $k \geq 3$. If $p$ is odd let $X_6$ be the set of stabilizers of conics in $U_3(q)$ and $X_7$ be the set of maximal subgroups of $U_3(q)$ which are isomorphic to any of $L_3(2)$, $A_6$, $A_6.2$, $A_7$, or the Hessian groups of order 216, 72, or 36. If $p=2$, set $X_6 = X_7 = \emptyset$. Finally, for $1 \leq i \leq 7$ let $x_i = \lvert X_i \cap \mathcal{C}\rvert$ and let $\displaystyle a_i = \max_{H \in X_i} \lvert H \cap \Omega \rvert$ if $X_i \neq \emptyset$, and $a_i = 0$ otherwise. By Theorems \ref{thm:Mitchell} and \ref{thm:Hartley}, $\displaystyle \bigcup_{i=1}^7 X_i$ contains all of the maximal subgroups of $U_3(q)$.

Now, since the stabilizer of each absolute point contains a unique Sylow $p$-subgroup, every subgroup in $X_1$ has exactly $q^3-1$ elements of $\Omega$. The stabilizer of a nonabsolute point $P$ has exactly $q^2-1$ elements of $\Omega$, namely the (nonidentity) elations in $U_3(q)$ whose centers are on $P^\perp$. Consequently $a_1 = q^3-1$ and $a_2 = q^2-1$. The stabilizer of a triangle contains no elements of $\Omega$ if $p \geq 5$, in which case $a_3 = 0$. If $p=2$, $0 \leq a_3 < 6(q+1)^2/\gcd(3,q+1)$ and $q \geq 8$. In both cases $a_3 <q^3-1$. Since we are assuming that $p \neq 3$, the stabilizer of an imaginary triangle contains no elements of $\Omega$ and $a_4 = 0$. A subgroup of $U_3(q)$ isomorphic to $U_3(q_0)$ or $PGU_3(q_0)$ has $q_0^6-1$ elements of $\Omega$, but if $q = q_0^k$ with $k \geq 3$, then $q_0^6-1 < q^3-1$, and we conclude that $a_5 <q^3-1$. If $p=2$ then $a_6 = a_7 = 0$. So suppose that $p\geq 5$. The stabilizer of a conic has order $q(q^2-1)$, so $a_6 \leq (q-1)(q^2-1) <q^3-1$. Now since $q\geq 7$, $q^3-1 \geq 342$, and the only subgroup in $X_7$ with more than 342 elements of order $p$ is $A_7$ in the case that $p=5$ which contains 504 elements of order 5. But if $p=5$ and $q \geq 7$, then in fact $q \geq 25$. So $504 < 15624 \leq q^3-1$ in this case. Consequently, we conclude that in all cases, $a_i <q^3-1 = a_1$ for $2 \leq i \leq 7$.

Since $\mathcal{C}$ is contained in $\displaystyle \bigcup_{i=1}^7 X_i$, $\mathcal{C}$ covers the $q^6-1$ elements of $\Omega$, and $a_4=0$, we must have
\[\sum_{\substack{i=1\\ i\neq 4}}^7a_i x_i \geq q^6-1\,.\]
Furthermore, since $a_i < q^3-1 = a_1$ for $2 \leq i \leq 7$, it follows that
\[(q^3-1)\sum_{\substack{i=1\\ i\neq 4}}^7 x_i \geq \sum_{\substack{i=1\\ i\neq 4}}^7a_i x_i \geq q^6-1\,.\]
Therefore $\displaystyle \sum_{\substack{i=1\\ i\neq 4}}^7 x_i \geq q^3+1$, and since $x_4 = q^3(q+1)^2(q-1)/3$ as noted earlier in the proof, we conclude that
\[\sigma(G) = \lvert \mathcal{C} \rvert = \sum_{i=1}^7 x_i \geq q^3+1+ q^3(q+1)^2(q-1)/3\,.\]

\section{Proof of Theorem \ref{thm:mainresult2}}
We conclude with a proof of Theorem \ref{thm:mainresult2}, which states that $\sigma(SU_n(q)) =\sigma(U_n(q))$. We will make use of the following lemmas:

\begin{lemma} \label{lem:equality of covering numbers of group and quotient}
Let $G$ be a finite noncyclic group and $N$ be a normal subgroup of $G$. Then $\sigma(G) = \sigma(G/N)$ if and only if there is a minimal cover of $G$ which consists of subgroups containing $N$. In particular, if $N$ is contained in every maximal subgroup of $G$, then $\sigma(G) = \sigma(G/N)$.
\end{lemma}

\begin{proof}
Note that the inequality $\sigma(G) \leq \sigma(G/N)$ holds for any normal subgroup $N$ of $G$ (this is Lemma 2 of \cite{Cohn1994}), and let $\eta$ be the canonical map from $G$ onto $G/N$. First suppose that there exists a minimal cover $\mathcal{C}$ of $G$ consisting of subgroups which contain $N$, and let $\mathcal{C}'=\{ \eta(H) : H \in \mathcal{C}\}$. Then $\mathcal{C}'$ is a cover of $G/N$ with $\lvert \mathcal{C}' \rvert = \lvert \mathcal{C} \rvert = \sigma(G)$. So $\sigma(G/N) \leq \lvert \mathcal{C}' \rvert =\sigma(G)$, and we conclude that $\sigma(G) = \sigma(G/N)$.

Now suppose that there does not exist a minimal cover of $G$ consisting only of subgroups which contain $N$. If $G/N$ is cyclic, then by convention $\sigma(G/N) = \infty$, and so it follows that $\sigma(G) < \sigma(G/N)$. If not, let $\mathcal{C}$ be a minimal cover of $G/N$, and let $\mathcal{C}' = \{\eta^{-1}(H) : H \in \mathcal{C}\}$. Then $\mathcal{C}'$ is a cover of $G$ and $\lvert \mathcal{C}' \rvert = \lvert \mathcal{C} \rvert =\sigma(G/N)$. But since every member of $\mathcal{C}'$ contains $N$, $\mathcal{C}'$ cannot be a minimal cover for $G$ and so $\sigma(G) < \sigma(G/N)$.

Finally, if $N$ is contained in every maximal subgroup of $G$, choose any minimal cover of $G$ consisting of maximal subgroups. Then every member of $\mathcal{C}$ contains $N$, so $\sigma(G) = \sigma(G/N)$ by the first part of the theorem.
\end{proof}

\begin{lemma} \label{lem: center or commutator}
Let $H$ be a maximal subgroup of a group $G$. Then $H$ contains the center or commutator subgroup of $G$.
\end{lemma}

\begin{proof}
If $H$ does not contain the center $Z$ of $G$, then $H \trianglelefteq HZ = G$, and $G/H = HZ/H \cong Z/(H \cap Z)$ is abelian, from which it follows that $H$ contains the commutator subgroup of $G$.
\end{proof}

\begin{lemma} \label{lem:perfect group}
Let $G$ be a finite nontrivial perfect group and $Z$ be the center of $G$. Then $\sigma(G) = \sigma(G/Z)$.
\end{lemma}

\begin{proof}
Since $G$ is perfect, no maximal subgroup contains the commutator subgroup of $G$, and therefore by the previous lemma, every maximal subgroup of $G$ contains the center of $G$. The desired result now follows from Lemma \ref{lem:equality of covering numbers of group and quotient}.
\end{proof}

To see that $\sigma(SU_n(q)) = \sigma(U_n(q))$, first note that $SU_2(2) \cong U_2(2) \cong S_3$, so $\sigma(SU_2(2)) = \sigma(U_2(2))  = 4$. In the case $n=2$ and $q=3$, $SU_2(3) \cong SL_2(3)$ and $U_2(3) \cong L_2(3)$, so it follows from the results in \cite{Bryce1999} that $\sigma(SU_2(3)) = \sigma(U_2(3)) = 5$. Note that $U_3(2)$ has the Klein 4-group as a homomorphic image. Therefore $SU_3(2)$ does as well, and so $\sigma(SU_3(2))  =\sigma(U_3(2)) = 3$. So we may assume that $(n,q) \notin \{(2,2),(2,3),(3,2)\}$ in which case $SU_n(q)$ is perfect, so that $\sigma(SU_n(q)) = \sigma(U_n(q))$ by Lemma \ref{lem:perfect group}.\\

We note that Lemma \ref{lem:perfect group} can be applied to other families of (usually) quasisimple groups. For example, applying it to the symplectic groups $Sp_{2m}(q)$ yields $\sigma(Sp_{2m}(q)) = \sigma(S_{2m}(q))$, which holds even in the cases $(m,q) \in \{(1,2),(1,3),(2,2)\}$ where $Sp_{2m}(q)$ is not perfect, since $Sp_2(2) \cong S_2(2)$, $Sp_4(2) \cong S_4(2)$, $Sp_2(3) \cong SL_2(3)$, $S_2(3) \cong L_2(3)$, and $\sigma(SL_2(3)) = \sigma(L_2(3)) = 5$.

\end{document}